\documentclass[11pt]{article}
\usepackage{amsmath, amssymb, amsfonts, amstext, amsthm, textcomp, enumerate}
\usepackage[mathscr]{euscript}
\usepackage{float}
\usepackage{booktabs}
\usepackage{mathtools}
\usepackage[left=20mm,top=0.5in,bottom=15mm]{geometry}
\usepackage{graphicx}
\usepackage{caption}
\usepackage{epstopdf}
\usepackage{longtable}
\usepackage[utf8]{inputenc}
\usepackage{color}
\usepackage{xcolor}
\usepackage{hyperref}
\usepackage{graphicx}
\usepackage{dcolumn}% Align table columns on decimal point
\usepackage{bm}% bold math
\usepackage{epstopdf}
\usepackage[english]{babel}
\usepackage{subfigure}
\usepackage{color}
\usepackage{ulem}

\newtheorem{thm}{Theorem}[section]
\newtheorem{cor}[thm]{Corollary}%[section]
\newtheorem{rem}[thm]{Remark}%[section]
\newtheorem{ex}[thm]{Example}%[section]

\newcommand{\RR}{\mathbb{R}}

\newfont{\bb}{msbm10}

\begin{document}
\allowdisplaybreaks
	\title{Threshold Graphs Allow Few Distinct Eigenvalues: A New Approach}
    \author{
    Jane Breen\thanks{Faculty of Science, Ontario Tech University, Oshawa, ON, Canada (Jane.Breen@ontariotechu.ca)}
       \and Shaun Fallat\thanks{Department of Mathematics and Statistics,
University of Regina, Regina, SK, Canada 
(sfallat@uregina.ca).}
    \and  Johnna Parenteau\thanks{Department of Mathematics and Statistics,
University of Regina, Regina, SK, Canada (Johnna.Parenteau@uregina.ca)}
    }
 \maketitle

\begin{abstract}
For any graph $G$, we associate a family of real symmetric matrices, $S(G)$, where for any $A \in S(G)$, the location of the nonzero off-diagonal entries of $A$ are governed by the adjacency structure of $G$. Let $q(G)$ represent the minimum number of distinct eigenvalues over all matrices in $S(G)$. In this work, we provide an alternative technique to establish that  $q(G) \leq 4$  for any threshold graph $G$ as presented in 
[L. Emilio Allem, C. Hoppen, J. Lazzarin, L. Siviero Sibemberg, F. Colman Tura, The minimum number of distinct eigenvalues of a threshold graph is at most 4,  Linear Algebra and its Applications, 726 (2025) 32–53]. In addition, we show that all connected threshold graphs admit a matrix having any four distinct eigenvalues. Further, such a matrix possesses an eigenvector for an interior eigenvalue having all nonzero coordinates. 
\end{abstract}

% edit keywords and MSC as needed

\noindent AMS-MSC: 05C50, 15A18, 15A29, and 15B57

 % $x \prec y$
 % $x^{\downarrow}$
 % $x^{\uparrow}$

\vspace{.5cm}
\section{Introduction}

%Let $G=(V(G), E(G))$ be a simple graph with the vertex set $V(G)=\{1,2, \ldots, n\}$ and edge set $E(G)$. The set $S(G)$ of symmetric matrices described by $G$ consists of the set of all symmetric $n \times n$ matrices $A=(a_{ij})$ such that for $i\neq j$, $a_{ij}\neq 0$ if and only if $ij\in E(G)$. 

Let $G=(V(G), E(G))$ be a simple graph with the vertex set $V(G)=\{1,2, \ldots, n\}$ and edge set $E(G)$. The set $S(G)$ of symmetric matrices described by $G$ consists of the set of all symmetric $n \times n$ matrices $A=(a_{ij})$ such that for $i\neq j$, $a_{ij}\neq 0$ if and only if $ij\in E(G)$. 

The study of the eigenvalues of matrices associated with graphs lies at the intersection of spectral graph theory, linear algebra, and combinatorics.  Rather than asking what spectrum arises from a prescribed matrix, we ask to determine which spectra are attainable among all matrices whose zero-nonzero pattern is governed by a given graph. This leads to the foundational `Inverse Eigenvalue Problem for Graphs (IEPG)', one of the central and most challenging problems in combinatorial matrix theory. The IEPG asks for a characterization of all all spectra that occur among the matrices in $S(G)$. Despite substantial progress over the past several decades, complete solutions are known only for a relatively small collection of graph families.

The problem of a minimum number of distinct eigenvalues has emerged as a central theme in recent research on the IEPG. Recent investigations on this parameter have also revealed deep connections with certain graph structures, providing new insights into how combinatorial constraints govern spectral complexity. Nevertheless, determining this number remains challenging even for many well-studied graph classes, and its computation continues to motivate substantial activity in inverse spectral graph theory.
To any graph $G$, we define the minimum number of eigenvalues of $G$ as
\[q(G)=\min\{q(A)\,:\, A\in S(G)\} \] where $q(A)$ represents the number of distinct eigenvalues of $A$. This parameter has garnered significant recent research attention (see, for example, \cite{border, brazil-2, FM, AACF, BFH, LJ2, qsmall, q-graphs} and the two related books \cite{JSbook, HLSbook}), and continues to be explored by researchers interested in the inverse eigenvalue problems for graphs (the so-called IEP-G). 
A significant portion of the research involving $q(G)$ has focused on graphs with $q(G)$ being small. The graphs with $q(G)=2$ have emerged as one of many unresolved, yet core investigations on this parameter. Thus, exhibiting graphs with $q(G)$ small, such as hypercubes, complete bipartite graphs, joins of connected graphs, has become a marked achievement on this subject.  

Along these lines, the minimum number of eigenvalues admitted by a threshold graph was explored by \cite{FM} and provided this value for many classes of threshold graphs and hinted that $q(G)$ for any threshold graph was much less than the number of vertices. In fact, a characterization is provided in \cite{FM} of the threshold graphs $G$ with trace equal to two and with $q(G)=4$. More recently, L. Allem et al. \cite{brazil-2} provided a proof that  any connected threshold graph $G$ satisfies $q(G) \leq 4$. 

Our work will also prove that any connected (and in fact disconnected too) threshold graph $G$ satisfies $q(G) \leq 4$. However, we use a completely different technique than the one presented in \cite{brazil-2}. The lengthy and constructive argument set out in \cite{brazil-2} systematically derives a matrix $A$ that fits a threshold graph and admits the four distinct eigenvalues of $A$ are $-\lambda, 0, \lambda, 2\lambda$, where $\lambda >0$. The argument we present features a much shorter explanation, exhibits the existence of a matrix achieving any four distinct real numbers as its spectrum, but is not explicitly constructive in nature. Our proof relies on a recent bordering mechanism that controls the change in multiplicities of the eigenvalues of a symmetric matrices embedded into a larger symmetric matrix. Along these lines are arguments are based on and verify the existence of so-called totally nonzero eigenvectors, and demonstrates a notion of spectrally arbitrariness in the designated four desired eigenvalues. 

We denote the spectrum of $A$, that is, the multi-set of (ordered) eigenvalues of $A$, by $\sigma(A) = \{ \lambda_1^{(n_1)}, \lambda_2^{(n_2)}, \ldots, \lambda_k ^{(n_k)}\}$, where $\lambda_1 < \lambda_2 < \cdots < \lambda_k$. Furthermore, we say $\lambda_1$ and $\lambda_k$ are the {\it extreme} eigenvalues of $A$ while $\lambda_2, \ldots, \lambda_{k-1}$ are called the {\it interior} eigenvalues of $A$. Here, a simple eigenvalue, $\lambda_i^{(1)}$, is abbreviated to $\lambda_i$. Given a graph $G$, the spectral invariant $q(G)$ is called the minimum number of distinct eigenvalues of a graph $G$. Let $N:=\{1,2,\dots,n\}$ and let $A:=(a_{ij})$ represent an $n \times n$ matrix. For an index set $\alpha\subseteq N,$ let the principal submatrix of $A$, lying in the rows and columns indicated by $\alpha$ be denoted by $A[\alpha].$ When $\alpha$ is empty, we define $A[\emptyset]:=1$.

The class of graphs considered here is known as {\it threshold graphs}. Threshold graphs can be characterized in many ways. Here we employ the notion that any threshold graph may be obtained through an iterative process which begins with an isolated vertex, and where, at each step, either a new isolated vertex is added, or a dominating vertex which is adjacent to all previous vertices, denoted by $G \vee K_1$, is added. As such, we can then represent a threshold graph on $n$ vertices using a binary sequence $(b_1, \ldots, b_n)$. Here $b_i$ is 0 if the vertex $v_i$ was added as an isolated vertex, and $b_i$ is 1 if $v_i$ was added as a dominant vertex. This representation has been called a {\it creation sequence} \cite{HSS}. For convenience, we use 0 as the first character of the string; it represents the first vertex of the graph. The number of characters 1 in the string, called the {\it trace} of the graph and denoted by $tr(G)$, indicates the number of dominating vertices in its construction \cite{Mer}. As we are mainly concerned with connected threshold graphs, we assume that $b_n =1$. A central goal in this work is to study the fewest eigenvalues allowed by a matrix in $S(G)$ when $G$ is a threshold graph.

\section{Background}

In this section, we provide the necessary background  from the literature in order to establish our key results. 

A vector in $\RR^n$ is called {\it totally nonzero} if each of its coordinates is nonzero (sometimes such a vector has been referred to as {\it nowhere zero}). Furthermore, we call an eigenvector $x$ corresponding to the eigenvalue $\lambda$ for a given matrix $A$, a \textit{totally nonzero interior eigenvector} whenever $x$ is totally nonzero and $\lambda$ is an interior eigenvalue of $A$. For ease of notation, we refer to an eigenvector $x$ with this property as TNIE, and that a matrix $A$ satisfies the TNIE property when it has an  interior eigenvalue with a corresponding totally nonzero eigenvector.

Referring to \cite{border}, an \textit{$r$-bordering} of a symmetric $n \times n$ matrix $A$ is any symmetric $(n+r)\times (n+r)$ matrix $B$, which contains $A$ as its $n \times n$ trailing principal submatrix of $B$. The following collection of results from \cite{border} concerns the spectrum of a 1-bordering of a given symmetric matrix and are very useful for the proofs of our main results in Section 3. The next result we present is critical to our analysis and represents a rather general result tracking the changes in the multiplicities of certain eigenvalues and 1-bordering.

\begin{thm}\label{thm:1-bordering} (Theorem 3.1, \cite{border})
Let $A$ be an $n\times n$ symmetric matrix and $A'$ a $1$-bordering of $A$. The following statements are equivalent: 

\begin{enumerate}
\item\label{thm:1-bordering1} $\mathcal{N}$ is the set of distinct eigenvalues $\lambda$ of $A'$ that satisfy $m_{A'}(\lambda)=m_{A}(\lambda)+1$, and 
$\mathcal{R}_0$ is the set of distinct eigenvalues $\lambda$ of $A$ that satisfy $m_{A'}(\lambda)=m_{A}(\lambda)-1$.
\item\label{thm:1-bordering2}
$A'=\left(
\begin{array}{cc}
 \alpha & {\bf b}^TU_0^T \\
 U_0{\bf b} & A \\
\end{array}
\right)$
where $k:=|\mathcal{R}_0|$, $U_0$ is an $n\times k$ matrix with $U_0^T U_0=I_{k}$, and $U_0^TAU_0$ is a $k \times k$ diagonal matrix $D_0$ with distinct eigenvalues equal to $\mathcal{R}_0$. Further, ${\bf b} \in \mathbb{R}^k$ is a nowhere zero vector
so that the matrix $$B=\left(
\begin{array}{cc}
 \alpha & {\bf b}^T \\
 {\bf b} & D_0 \\
\end{array}
\right)$$
has eigenvalues $\mathcal{N}$. 
\end{enumerate}
If the above hold, then
$A'$ is similar to a matrix of the form $D_{\mathcal{N}} \oplus D_1$ for some diagonal matrix $D_1$ via an orthogonal similarity using
\begin{equation}\label{eq:Wsimilarity}
W=\left(
\begin{array}{cc}
 {\bf v}^T & 0 \\
 U_0V_0 & U_1 \\
\end{array}
\right),
\end{equation}
where 
%${\bf v}\in \bR^{|\mathcal{N}|}$ and 
$V=\left(
\begin{array}{c}
 {\bf v}^T \\
 V_0
\end{array}
\right)\in \RR^{|\mathcal{N}|\times |\mathcal{N}|}$ is an orthogonal matrix that
satisfies $V^T BV=D_{\mathcal{N}}$, and $U=\left(
\begin{array}{cc}U_0 & U_1\end{array}
\right)$ is an orthogonal matrix that satisfies $U^TAU=D_{0} \oplus D_1$. 
\end{thm}

\begin{rem}
    Referring to Theorem \ref{thm:1-bordering}, we consider the specific case when  $|\mathcal R_0|=1$, and thus the matrix $B$ is then a $2 \times 2$ matrix and $b$ is a nonzero scalar. The set $\mathcal{N}$ consists of two eigenvalues each with totally nonzero eigenvectors, which are the columns of the matrix $V$ above. Upon closer inspection of the first $|\mathcal{N}|$ columns of the matrix $W$ that diagonalizes this 1-bordering of $A$ (namely $A'$), it follows that these columns are totally nonzero whenever we have the column $U_0$ being totally nonzero. This observation is essential for our proofs in Section 3.
\end{rem}

The following result illustrates the connection between 1-bordering as outlined in the previous theorem and its effect on eigenvalue multiplicities, beyond the classical conditions known as the {\it Cauchy's interlacing inequalities}, see, for example, \cite{HJ1}. If two sequences of real numbers satisfy Cauchy's interlacing inequalities with strict inequalities holding everywhere, we say these sequences {\it strictly interlace}.

\begin{cor}\label{lem:1-border} (Corollary 3.2, \cite{border})
Let $A$ be a symmetric matrix $n\times n$; $\mathcal{R}$ the set of distinct eigenvalues of $A$ and $\mathcal{R}_0\subseteq \mathcal{R}$. If $\mathcal{N}$ is any set of $|\mathcal{R}_0|+1$ distinct real numbers which strictly interlace $\mathcal{R}_0$, then there is a $1$-bordering $A'$ of $A$ so that for $\lambda \in \RR$,
\begin{align*}
m_{A'}(\lambda)=\begin{cases}m_{A}(\lambda)-1 &\text{if }\lambda \in \mathcal{R}_0, \\
m_{A}(\lambda)+1 &\text{if }\lambda \in \mathcal{N}, \\
m_{A}(\lambda) &\text{otherwise},\end{cases}
\end{align*}
where $m_{A'}(\lambda)=0$ means that $\lambda$ is not an eigenvalue of $A'$. 
\end{cor}

If we want a $1$-bordering of the matrix $A \in S(G)$ to produce a matrix $A' \in S(K_1 \vee G)$, then we need $U_0{\bf b}$ to have no zero entries in Theorem \ref{thm:1-bordering} above. This will happen for most choices of ${\bf b}$, unless $U_0$ contains a zero row, or equivalently, unless the eigenvectors corresponding to the eigenvalues in $\mathcal R_0$ all have a zero entry in the same position. The next results consider the case $|\mathcal R_0|=1$. We call an eigenvalue of a symmetric matrix \textit{extreme} if it is the smallest or the largest eigenvalue of that matrix.

\begin{cor}\label{cor:join-with-K1} (Corollary 4.1, \cite{border})
Suppose $G$ is a non-empty graph and there exists an $A\in S(G)$ with a nowhere zero eigenvector associated with some eigenvalue $\lambda$ of $A$. Then there exists a $1$-bordering $A'$ of $A$ in $S(K_1 \vee G)$ so that:
\begin{itemize}
\item $q(A')=q(A)+1$ if $\lambda$ is an extreme eigenvalue, 
\item $q(A')=q(A)$ if $\lambda$ is not simple nor an extreme eigenvalue,  
\item $q(A')=q(A)-1$ if $\lambda$ is simple and not an extreme eigenvalue. 
\end{itemize}
\end{cor}

The next two results from \cite{MS} and and \cite{JDS} (extending works from the classical work in \cite{Parter}) will be used in the argument of our main result in order to lay the foundation for the existence of totally nonzero eigenvectors. 

\begin{thm}(Theorem 4.3, \cite{MS}) 
For a given connected graph $G$ on $n$ vertices, and given distinct eigenvalues $\lambda_1, \lambda_2, \ldots, \lambda_n$, there exists a real symmetric matrix $A$ whose graph is $G$ and its eigenvalues are $\lambda_1, \lambda_2, \ldots, \lambda_n$ such that none of the eigenvectors of $A$ has a zero entry. 
\label{MonShad}
\end{thm}

\begin{thm} \cite{JDS}
For a tree $T$, $A \in S(T)$, and $\lambda \in \sigma(A) \cap \sigma(A(v))$. Then: 
\vspace*{-0.5cm}
\begin{enumerate}[(1)] 
		\item there is a vertex $u \in T$ where $m_{A(u)}(\lambda) = m_A(\lambda) +1$. 
		\item if $m_A(\lambda) \geq 2$, then the vertex $u$ in (1) can be chosen so that $deg(u) \geq 3$, and $\lambda$ is an eigenvalue of at least three branches at $u$. 
		\item if $m_A(\lambda) =1$, then $u$ can be chosen with $deg(u) \geq 2$, and $\lambda$ is an eigenvalue of at least two branches at $u$. 
	\end{enumerate}
\label{expart}\end{thm} 

\begin{rem}
One consequence of the Parter-Wiener Theorem for trees (see also Theorem \ref{expart}) is the existence of a nonzero coordinate in every eigenvector that corresponds to a so-called Parter vertex (the designated vertex $u$ in (1) of Theorem \ref{expart}). It is known (see \cite[Prop. 21.]{KS}) that the coordinate of any eigenvector corresponding to $\lambda$, as in (1) of Theorem \ref{expart},  must be equal to zero. We illustrate this phenomenon in the specific case when $G$ is a star on $n$ vertices, which is particularly useful for our purpose in the next section. Let $A \in S(G)$ where $A = \left [ \begin{array}{c|c}
y & x^T \\ \hline
x & 0 \\
\end{array} \right]$
with $\sigma(A) = \{ \alpha, 0^{(n-2)}, \beta\}$. For each null vector $z$ that satisfies $Az = 0$, $z_1 = 0$. 

To see this, observe that ${\rm rank}\{A\} = 2$. Since $A \in S(G)$, we have $A = \left [ \begin{array}{c|c}
y & x^T \\ \hline
x & D \\
\end{array} \right ]$ where $D = {\rm diag}\{ d_1, d_2, \ldots, d_{n-1}\}$ and $d_j \neq 0$ for some $1 \leq j \leq n$. Consider the principal submatrix based on $\{ 1, j, n \}$ if $j <n$; otherwise, take $\{ 1, 2, n\}$. This $3 \times 3$ principal submatrix has full rank, which contradicts ${ \rm rank}\{A\} = 2$. Hence, $D = 0$, and $A = \left [ \begin{array}{c|c}
y & x^T \\ \hline
x & 0 \\
\end{array} \right ]$. Assume $z$ is a non-trivial null vector for $A$. Then, partition $z$ conformally with $A$; namely, $z = \left [ \begin{array}{c}
z_1 \\ \hline
z_2 \\
\end{array} \right ]$ where $z_2 \in \mathbb{R}^{n-1}$. Consider $Az = 0$, then $ \left [ \begin{array}{c|c}
y & x^T \\ \hline
x & 0 \\
\end{array} \right] \left[ \begin{array}{c}
z_1 \\ \hline
z_2 \\
\end{array} \right ] = 0$ if and only if $\left [ \begin{array}{c}
yz_1 + x^Tz_2 \\ \hline
z_1x \\
\end{array} \right ] = 0$. Since $x$ is totally nonzero, $z_1 =0$. 
\label{PWrem}
\end{rem}

\section{Main Results}

In this section, we present an alternative proof that threshold graphs have at most four distinct eigenvalues. In addition, we establish a technique that proves threshold graphs allow a matrix with four distinct eigenvalues, and having a totally nonzero interior eigenvector. Moreover, we  ascertain that threshold graphs admit matrices whose distinct eigenvalues consist of any four arbitrary real numbers.

\begin{thm}
    If $G$ is a connected threshold graph on $n \geq 4$ vertices, then there exists  $A\in S(G)$ such that $q(A)=4$, and $A$ satisfies the TNIE property. In particular, $q(G) \leq 4$.
    \label{main} 
\end{thm}

\begin{proof} 
Suppose $G$ is a graph on $n = 4$ vertices. Then, by Theorem \ref{MonShad}, for any four distinct real numbers $\lambda_1, \lambda_2, \lambda_3,$ and $\lambda_4$, there exists $A \in S(G)$ such that $\sigma(A) = \{\lambda_1, \lambda_2, \lambda_3, \lambda_4 \}$ and their corresponding eigenvectors are totally nonzero. Hence, $q(A) = 4$, and $A$ has the TNIE property. Now, assume $n \geq 4$, and $tr(G) = 1$. Then $G = (0, \cdots, 0, 1)$, which is a star. It is easy to construct a matrix $A \in S(G)$ where $\sigma(A) = \{ \lambda_1, \lambda_2^{(n-3)}, \lambda_3, \lambda_4 \}$ for some $\lambda_1, \ldots, \lambda_4$ (possibly $\lambda_2$ and $\lambda_3$ may be swapped). Furthermore, it follows that any eigenvector corresponding to $\lambda_1, \lambda_3,$ or $ \lambda_4$ must be totally nonzero (see Remark \ref{PWrem} for details). Thus, $q(A) = 4$, and the matrix $A$ has the TNIE property. 

Now, assume that the result holds for all connected threshold graphs on at least four vertices with $tr(G) < k$, and suppose that $tr(G) = k \geq 2$. Then, $G = (0, \cdots, 1, \underset{s}{\underbrace{0, \cdots, 0}}, 1)$ with $s\geq 0$. Reading the creation sequence from right to left, we will induct on the subgraph that ends on the second ``1"; that is, $H = (0, \cdots, 1)$. If $H$ is a subgraph on at least four vertices, then by induction, there is a matrix $B \in S(H)$ with the spectrum $\sigma(B) = \{ \lambda_1^{(n_1)}, \lambda_2^{(n_2)}, \lambda_3^{(n_3)}, \lambda_4^{(n_4)} \}$ where either $\lambda_2$ or $\lambda_3$ has a TNIE eigenvector. Without loss of generality, assume $\lambda_2$ has a TNIE eigenvector. Construct a new matrix $A = \left[\begin{array}{c|c}
B & 0 \\ \hline 
0 & \lambda_2 I_s \\
\end{array} \right]$, which corresponds to the subgraph $ H' = (0, \cdots, 1, \underset{s}{\underbrace{0, \cdots, 0}})$. This new matrix $A$ has the spectrum $\sigma(A) = \{ \lambda_1^{(n_1)}, \lambda_2^{(n_2 +s)}, \lambda_3^{(n_3)}, \lambda_4^{(n_4)} \}$, and observe that padding the matrix by zeros in this way allows only the eigenvalue $\lambda_2$ to retain a corresponding TNIE eigenvector. Since $A$ now has a totally nonzero eigenvector, we can invoke Corollary \ref{lem:1-border} and create a $1$-bordering $A'$ of $A$, where $A' \in S(K_1 \vee H')$ and $K_1 \vee H' \cong G = (0, \cdots, 1, \underset{s}{\underbrace{0, \cdots, 0}}, 1)$, such that the multiplicities of $\lambda_1$ and $\lambda_3$ increase by one (set $\mathcal{N} = \{\lambda_1, \lambda_3\}$ in Corollary \ref{lem:1-border}), and the multiplicity of $\lambda_2$ decreases by 1 (set $\mathcal{R}_0 = \{\lambda_2\}$ in Corollary \ref{lem:1-border}). Hence, $\sigma(A') = \{ \lambda_1^{(n_1+1)}, \lambda_2^{(n_2 +s -1)}, \lambda_3^{(n_3 +1)}, \lambda_4^{(n_4)} \}$, so that $q(A') = 4$ for $A' \in S(G) $. Corollary \ref{lem:1-border} also ensures that $\lambda_1$ and $\lambda_3$ each gain a TNIE eigenvector, and since $\lambda_3$ is not an extreme eigenvalue, $A'$ has the TNIE property, as desired.
In the special case when $s=0$ and $n_2=1$, selecting $\mathcal{R}_0=\{\lambda_2\}$ means we choose, for example, $\mathcal{N}= \{\mu, \lambda_3\}$, where $\lambda_1<\mu< \lambda_3$. Thus, resulting in a 1-bordering $A'$ of $A$ with $A' \in S(G)$ with $\sigma(A') = \{ \lambda_1^{(n_1+1)}, \mu, \lambda_3^{(n_3 +1)}, \lambda_4^{(n_4)} \}$, so that $q(A') = 4$.

If $H$ is a subgraph on less than four vertices, then we consider two cases: 

\textbf{Case 1:} If $H$ has 2 vertices, then $H = (0,1)$ and $G = (0, 1, \underset{s}{\underbrace{0, \cdots, 0}}, 1)$ where $s \geq 1$. From Theorem \ref{MonShad}, there exists a matrix $B \in S(H)$ with the spectrum $\sigma(B) = \{ \lambda_1, \lambda_3\}$ where either eigenvalue has a TNIE eigenvector. Without loss of generality, consider $\lambda_2$ and its TNIE eigenvector. Construct a new matrix $A = \left[\begin{array}{c|c}
B & 0 \\ \hline 
0 & \lambda_3 I_s \\
\end{array} \right]$, which corresponds the subgraph $ H' = (0, 1, \underset{s}{\underbrace{0, \cdots, 0}})$. This new matrix $A$ has the spectrum $\sigma(A) = \{ \lambda_1, \lambda_3^{(s+1)}\}$, and the eigenvalue $\lambda_3$ retains a corresponding TNIE eigenvector. Invoking Corollary \ref{lem:1-border}, we can create a 1-bordering $A'$ of $A$ that gives rise to two new eigenvalues $\lambda_2$ and $\lambda_4$  (set $\mathcal{N} = \{\lambda_2, \lambda_4\}$ in Corollary \ref{lem:1-border}), and the multiplicity of $\lambda_3$ decreases by 1 (set $\mathcal{R}_0 = \{\lambda_2\}$ in Corollary \ref{lem:1-border}). Hence, $\sigma(A') = \{ \lambda_1, \lambda_2, \lambda_3^{(s)}, \lambda_4 \}$, so that $q(A') = 4$ for $A' \in S(G)$. Moreover, $A'$ has the TNIE property. 

\textbf{Case 2:} If $H$ has 3 vertices, then $G$ is of the form $G = (0, 1, 1, \underset{s}{\underbrace{0, \cdots, 0}}, 1)$
or $G = (0, 0,  1, \underset{s}{\underbrace{0, \cdots, 0}}, 1)$. In either construction, if $s = 0$, then $n = 4$, which allows us to utilize Theorem \ref{MonShad} and induct using the same argument as above. Assume $s \geq 1$. Since the two constructions follow the same argument, we will illustrate only the argument for $G = (0, 1, 1, \underset{s}{\underbrace{0, \cdots, 0}}, 1)$ where $H = (0, 1, 1)$.  From Theorem \ref{MonShad}, there exists a matrix $B \in S(H)$ with the spectrum $\sigma(B) = \{ \lambda_1, \lambda_2, \lambda_4\}$ where all eigenvalues have a TNIE eigenvector. Consider $\lambda_2$ and its TNIE eigenvector. Construct a new matrix $A = \left[\begin{array}{c|c}
B & 0 \\ \hline 
0 & \lambda_2 I_s \\
\end{array} \right]$, which corresponds to the subgraph $ H' = (0, 1, 1, \underset{s}{\underbrace{0, \cdots, 0}})$. This new matrix $A$ has the spectrum $\sigma(A) = \{ \lambda_1, \lambda_2^{(s+1)}, \lambda_4\}$, and the eigenvalue $\lambda_2$ retains a corresponding TNIE eigenvector. Invoking Corollary \ref{lem:1-border}, we can create a 1-bordering $A'$ of $A$ that gives rise to a new eigenvalue $\lambda_3$ and increases the multiplicity of $\lambda_1$ by one  (set $\mathcal{N} = \{\lambda_1, \lambda_3\}$ in Corollary \ref{lem:1-border}), and the multiplicity of $\lambda_2$ decreases by 1 (set $\mathcal{R}_0 = \{\lambda_2\}$ in Corollary \ref{lem:1-border}). Hence, $\sigma(A') = \{ \lambda_1^{(2)}, \lambda_2^{(s)}, \lambda_3, \lambda_4 \}$, so that $q(A') = 4$ for $A' \in S(G)$. Moreover, $A'$ has the TNIE property. 
\end{proof} 

 A benefit of this technique is that it provides a new perspective on how to tackle the characterization of $q(G)$ for some threshold graphs that have eluded detection using other established methods. An example that best illustrates this is the threshold graph $G = (0, 0, 0, 1, 1, 0, 0, 1)$ from \cite{FM} whose $q(G)$ is currently unknown. 

\begin{ex}
Consider the threshold graph $G = (0, 0, 0, 1, 1, 0, 0, 1)$. We begin by using our technique on the subgraph $H = (0, 0, 0, 1)$. From Theorem \ref{MonShad}, there exists a matrix $B \in S(H)$ such that $\sigma(B) = \{ \lambda_1, \lambda_2, \lambda_3, \lambda_4\}$ each of whose eigenvectors are totally nonzero. Invoking Corollary \ref{lem:1-border}, set $\mathcal{R}_0 =\{ \lambda_2\}$ and $\mathcal{N} = \{ \lambda_1, \lambda_3\}$ to create a one bordering $B'$ of $B$ that has spectrum $\sigma(B') = \{ \lambda_1^{(2)}, \lambda_3^{(2)}, \lambda_4\}$ for $H' = (0, 0, 0, 1, 1)$. Observe that the eigenvalues $\lambda_1$ and $ \lambda_3$ have totally nonzero eigenvectors. Now, consider $\lambda_3$ and its TNIE eigenvector. Construct a new matrix $B'' = \left[\begin{array}{c|c}
B' & 0 \\ \hline 
0 & \lambda_3 I_2 \\
\end{array} \right]$, which corresponds the subgraph $ H'' = (0, 0, 0, 1, 1, 0, 0)$. This new matrix $B''$ has the spectrum $\sigma(B'') = \{ \lambda_1^{(2)}, \lambda_3^{(4)}, \lambda_4\}$, and the eigenvalue $\lambda_3$ retains a corresponding TNIE eigenvector. Invoking  Corollary \ref{lem:1-border}, set $\mathcal{R}_0 =\{ \lambda_3\}$ and $\mathcal{N} = \{ \lambda_1, \lambda_4\}$ to create a one bordering $A$ of $B''$ that has spectrum $\sigma(A) = \{ \lambda_1^{(3)}, \lambda_3^{(3)}, \lambda_4^{(2)}\}$. This shows that there is a matrix $A$ that achieves $q(A) = 3$, so $q(G) = 3$ since $G$ contains a pendant vertex (see \cite{AACF}). 
\end{ex}

Interestingly, locating matrices that produce a low value of $q(A)$ using this method is much more difficult than it appears. The example and remarks presented below describe some of the challenges associated with our technique. 

\begin{ex}
Consider the graph $G = (0, 0, 0, 1, 1, 1)$. This threshold graph is known to have $q(G) = 2$ from \cite{FM}. Can we use the bordering technique to deduce $q(G) = 2$ as well? 

Inspect the subgraph $H = (0, 0, 0, 1)$. From Theorem \ref{MonShad}, there exists a matrix $B \in S(G)$ such that $\sigma(B) = \{\lambda_1, \lambda_2, \lambda_3, \lambda_4 \}$ where every eigenvalue has a corresponding totally nonzero eigenvector. Take $\mathcal{R}_0 = \{ \lambda_2\}$, and $\mathcal{N} = \{ \lambda_1, \lambda_3\}$. Invoking Corollary \ref{lem:1-border} and 1-bordering  $B$ to create a new matrix $B'$, the spectrum of the subgraph $H' = (0, 0, 0, 1, 1)$ is $\sigma(B') = \{\lambda_1 ^{(2)}, \lambda_3^{(2)}, \lambda_4 \}$ where eigenvalues $\lambda_1$ and $\lambda_3$ have totally nonzero eigenvectors. Now, set $\mathcal{R}_0 = \{ \lambda_3\}$, and $\mathcal{N} = \{ \lambda_1, \lambda_4\}$. Invoking Corollary \ref{lem:1-border} again and 1-bordering $B'$ to create a new matrix $A$, the spectrum of the subgraph $G = (0, 0, 0, 1, 1, 1)$ 
is $\sigma(A) = \{\lambda_1 ^{(3)}, \lambda_3^{(1)}, \lambda_4^{(2)} \}$ where the eigenvalues $\lambda_1$ and $\lambda_4$ have totally nonzero eigenvectors. Thus, we are stuck with a matrix $A$ having three distinct eigenvalues not two. 

However, if we select the set $\mathcal{R}_0$ differently, we can achieve different results. For example, if we inspect the subgraph $H = (0, 0, 0, 1)$ with corresponding matrix $B \in S(G)$ and spectrum $\sigma(B) = \{\lambda_1, \lambda_2, \lambda_3, \lambda_4 \}$ again. Take $\mathcal{R}_0 = \{ \lambda_2, \lambda_3\}$, and $\mathcal{N} = \{ \lambda_1, \lambda_{2.5}, \lambda_4\}$. Invoking Corollary \ref{lem:1-border} and 1-bordering  $B$ to create a new matrix $B'$, the spectrum of the subgraph $H' = (0, 0, 0, 1, 1)$ is $\sigma(B') = \{\lambda_1 ^{(2)}, \lambda_{2.5}, \lambda_4^{(2)}\}$ where all eigenvalues have totally nonzero eigenvectors. Now, set $\mathcal{R}_0 = \{ \lambda_{2.5}\}$, and $\mathcal{N} = \{ \lambda_1, \lambda_4\}$. Invoking Corollary \ref{lem:1-border} again and 1-bordering $B'$ to create a new matrix $A$, the spectrum of the subgraph $G = (0, 0, 0, 1, 1, 1)$ 
is $\sigma(A) = \{\lambda_1 ^{(3)}, \lambda_4^{(3)} \}$.
\end{ex}

\begin{rem}
Observe that Theorem \ref{main} may not hold if $q(A) \leq 3$. Consider the case where $G = (0, \cdots, 0, 1)$ is a star with $A \in S(G)$ such that $q(A) = 3$. Then $\sigma(A) = \{ \lambda_1, \lambda_2^{(n-2)},\lambda_3 \}$. However, the root (or center) of a star is the so-called Parter vertex, and $\lambda_2$ for the only multiple eigenvalue. Furthermore, an application of the Parter-Wiener theorem and a basic analysis of the eigenvalue-eigenvector equation for the eigenvalue $\lambda_2$ reveals that any eigenvector corresponding to $\lambda_2$ will have a zero entry corresponding to the root vertex (see Remark \ref{PWrem}). Hence, any matrix in $S(G)$ with spectrum  $\{ \lambda_1, \lambda_2^{(n-2)},\lambda_3 \}$ cannot have the TNIE property.  
\end{rem}

We are now in a position to show that our technique is sufficient to ascertain threshold graphs with four distinct eigenvalues that are in a sense spectrally arbitrary. 

\begin{thm}
    If $G$ is a connected threshold graph on $n \geq 4$ vertices and $\lambda_1< \lambda_2 < \lambda_3 < \lambda_4$ are any four distinct real numbers, then there is a matrix $A\in S(G)$ such that the distinct eigenvalues of $A$ are 
     $\{ \lambda_1, \lambda_2, \lambda_3, \lambda_4\}$. In particular, connected threshold graphs admit a matrix with any four distinct real numbers as its eigenvalues.
     \label{main2}
\end{thm}

\begin{proof}
Let $G$ be a connected threshold graph, and $\lambda_1< \lambda_2 < \lambda_3 < \lambda_4$ be any four distinct real numbers. Suppose $G$ has $n = 4$ vertices. Then, by Theorem \ref{MonShad}, there exists $A \in S(G)$ such that $\sigma(A) = \{\lambda_1, \lambda_2, \lambda_3, \lambda_4 \}$. Now, assume $n \geq 4$, and $tr(G) = 1$. Then, $G \cong (0, \cdots, 0, 1)$, which is a star. As noted previously, the multiplicity list $(1,n-3,1,1)$ or $(1,1,n-3,1)$  is allowed for a star, from which it follows by \cite[Cor. 11]{JW}, that any list of four real numbers can be achieved as the eigenvalues of a matrix for the star with this multiplicity list. Hence, there exists $A \in S(G)$ where $\sigma(A) = \{ \lambda_1, \lambda_2^{(n-3)}, \lambda_3, \lambda_4 \}$.

Now, assume the result holds for all connected threshold graphs on at least four vertices with $tr(G) < k$, and suppose $tr(G) = k \geq 2$. Then, $G \cong (0, \cdots, 1, \underset{s}{\underbrace{0, \cdots, 0}}, 1)$, with $s\geq 0$. Reading the creation sequence from right to left, we will induct on the subgraph that ends on the second ``1"; that is, $H \cong (0, \cdots, 1)$. If $H$ is a subgraph on at least four vertices, then by induction, there is a matrix $B \in S(H)$ with the spectrum $\sigma(B) = \{ \lambda_1^{(n_1)}, \lambda_2^{(n_2)}, \lambda_3^{(n_3)}, \lambda_4^{(n_4)} \}$ where either $\lambda_2$ or $\lambda_3$ has a TNIE eigenvector. Without loss of generality, assume $\lambda_2$ has a TNIE eigenvector. Construct a new matrix $A = \left[\begin{array}{c|c}
B & 0 \\ \hline 
0 & \lambda_2 I_s \\
\end{array} \right]$, which corresponds the subgraph $ H' \cong (0, \cdots, 1, \underset{s}{\underbrace{0, \cdots, 0}})$. This new matrix $A$ has the spectrum $\sigma(A) = \{ \lambda_1^{(n_1)}, \lambda_2^{(n_2 +s)}, \lambda_3^{(n_3)}, \lambda_4^{(n_4)} \}$, and observe that padding the matrix by zeros in this way allows only the eigenvalue $\lambda_2$ to retain a corresponding TNIE eigenvector. Since $A$ now has a totally nonzero eigenvector, we can invoke Theorem \ref{cor:join-with-K1} and create a $1$-bordering $A'$ of $A$, where $A' \in S(K_1 \vee H')$ and $K_1 \vee H' \cong G = (0, \cdots, 1, \underset{s}{\underbrace{0, \cdots, 0}}, 1)$, such that the multiplicities of $\lambda_1$ and $\lambda_3$ increase by one (set $\mathcal{N} = \{\lambda_1, \lambda_3\}$ in Corollary \ref{lem:1-border}), and the multiplicity of $\lambda_2$ decreases by 1 (set $\mathcal{R}_0 = \{\lambda_2\}$ in Corollary \ref{lem:1-border}). Hence, $\sigma(A') = \{ \lambda_1^{(n_1+1)}, \lambda_2^{(n_2 +s -1)}, \lambda_3^{(n_3 +1)}, \lambda_4^{(n_4)} \}$, so that $q(A') = 4$ for $A' \in S(G) $. 
In the special case when $s=0$ and $n_2=1$, we may apply the process of cloning a dominating vertex with an edge as presented in \cite{qsmall}. Thus resulting in a 1-bordering $A'$ of $A$ with $A' \in S(G)$ that satisfies $q(A') = 4$ with distinct eigenvalues $\lambda_1, \lambda_2, \lambda_3, \lambda_4.$

If $H$ is a subgraph on less than four vertices, then these cases will follow by observing the following graphs are known to be spectrally arbitrary for any given allowed multiplicity list:  $K_2 = (0,1)$, $K_3 = (0, 1, 1)$, and $P_3 = (0, 0, 1)$. Using this information as well as specifying the sets $\mathcal{N}$ and $\mathcal{R}_0$ as illustrated in the previous proof, the result that all threshold graphs with four distinct eigenvalues are spectrally arbitrary, relative to allowed multiplicity lists, follows immediately. 
\end{proof}

Previous literature on the characterization of $q(G) \leq 4$ for threshold graphs provides a construction for the matrices that achieve the minimum number of distinct eigenvalues (see \cite{brazil-2}). While our technique is not constructive in this manner, our technique and arguments prove an analogous result using plainer methods whilst gaining threshold graphs with four distinct eigenvalues have a TNIE eigenvector and allow any four real numbers as eigenvalues at the same time. Moreover, Theorems \ref{main} and \ref{main2} can be extended to threshold graphs with isolated vertices, as well as the disjoint union of threshold graphs. For example, let $G=(0, b_1, \cdots, b_s, 1, 0, \cdots,0)$ on at least four vertices. By Theorem \ref{main}, there exists a matrix $A \in S(H)$, where $H=(0, b_1, \cdots, b_s, 1)$ with $q(A)=4$ and having the TNIE property. By appending an appropriately chosen diagonal block to $A$ we obtain a matrix in $S(G)$ with four distinct eigenvalues and satisfies the TNIE property. On the other hand, suppose $G_1$ and $G_2$ are two vertex-disjoint connected threshold graphs, each on at least four vertices. The proof of Theorem \ref{main} can be adapted to construct a matrix with four distinct eigenvalues and the TNIE property, while simultaneously prescribing the location of the interior eigenvalue that admits a totally nonzero eigenvector. Combining this observation with Theorem \ref{main2}, it follows that the disconnected graph $G_1 \sqcup G_2$ admits a matrix with four distinct eigenvalues and possesses the TNIE property. Consequently, the class of graphs admitting a matrix with four distinct eigenvalues and the TNIE property includes threshold graphs with isolated vertices and disjoint unions of connected threshold graphs. In the case when $G_1$ and $G_2$ each have less than four vertices, we can show there is a matrix $A$ such that  $q(A) = 3$, and $A$ has the TNIE property for by selecting a common interior eigenvalue from both spectra of $G_1$ and $G_2$. When $G_1$ and $G_2$ both have exactly two vertices, it can be shown that $q(G) = 2$, and $G$ has the TNIE property.

%\begin{thm}[FM21]\label{T=2}
% Let $G\cong (\underbrace{0,\ldots,0}_{k_1},1, \underbrace{0,\ldots,0}_{k_2},1)$ be a threshold graph of order $n\ge 5$ with $k_1\ge 1$ and $k_2\ge 0$. Then,
%$$q(G)=\begin{cases}
%3~~~~\mbox{if}\, k_1\in \{1,\,2\} ~\mbox{or}~ k_2\in \{0,\,1\},\\
%4~~~~\mbox{otherwise}.
%\end{cases}$$
% \end{thm}

%\begin{thm}
  %  Suppose $G$ is a connected threshold graph on at least 4 vertices. Then $q(G)=4$ if and only if $G\cong (\underbrace{0,\ldots,0}_{k_1},1, \underbrace{0,\ldots,0}_{k_2},1)$ where $k_1 \geq 3$ and $k_2 \geq 2$.
%\end{thm}

\section*{Acknowledgments}
J. Breen is supported in part by an NSERC Discovery Research Grant, Application No.: RGPIN-2021-03775. S.M.\ Fallat is supported in part by an NSERC Discovery Research Grant, Application No.: RGPIN-2025-05272.
 J. Parenteau is supported in part by an NSERC PGSD Award, Application No.: PGS D-599971-2025.

\end{document}